\documentclass[11pt]{preprintstyle}
\usepackage{amsfonts}
\usepackage{graphicx}

\begin{document}

\title{
Stability and asymptotic behavior of periodic traveling wave solutions of 
viscous conservation laws in several dimensions\footnote{Preprint: October 7, 2008, Modified: October 7, 2008}}

\author{Myunghyun Oh\thanks{Research of M.O. was partially supporte
under NSF grant No. DMS-0204072.}
and Kevin Zumbrun\thanks{Research of K.Z. was partially supported
under NSF grants no. DMS-0070765 and DMS-0300487.}}
\institute{Department of Mathematics, University of Kansas, USA;
moh@math.ku.edu}
\institute{Department of Mathematics, Indiana University, USA;
kzumbrun@indiana.edu}

\maketitle

\begin{abstract} 
Under natural spectral stability assumptions 
motivated by previous investigations of the associated spectral stability
problem,
we determine sharp $L^p$ estimates on the linearized solution operator 
about a multidimensional planar periodic wave
of a system of conservation laws with viscosity,
yielding linearized $L^1\cap L^p\to L^p$ stability
for all $p \ge 2$ and dimensions $d \ge 1$ and 
nonlinear $L^1\cap H^s\to L^p\cap H^s$ stability
and $L^2$-asymptotic behavior for $p\ge 2$ and $d\ge 3$.
The behavior can in general be rather complicated, involving
both convective (i.e., wave-like) and diffusive effects. 
%NOTE: rather than pure
%diffusion as for reaction-diffusion and Cahn--Hilliard equations.
\end{abstract}

%\tableofcontents

%TODO:
%\keywords{stability, WKB, periodic traveling-waves}
%\primclass{35}
%\secclasses{34}

%	definitions

\newcommand{\curl}{\hbox{ \rm curl }}
\newcommand{\rmO}{\mathrm{O}}
\newcommand{\rmo}{\mathrm{o}}
\newcommand{\rmd}{\mathrm{d}}
\newcommand{\rme}{\mathrm{e}}
\newcommand{\rmi}{\mathrm{i}}

\newcommand{\R}{\mathbb{R}}
\newcommand{\C}{\mathbb{C}}
\newcommand{\Z}{\mathbb{Z}}
\newcommand{\N}{\mathbb{N}}

\newcommand{\op}{{\cal L}}
\newcommand{\Or}{{\cal O}}
\newcommand{\oOr}{{o}}
\newcommand{\setP}{{\cal P}}
\newcommand{\CalA}{{\cal A}}
\newcommand{\CalD}{{\cal D}}
\newcommand{\e}{\hbox {\rm e}}
\newcommand{\vp}{\varphi}
\newcommand{\bu}{\bar{u}}
\newcommand{\du}{\dot{u}}
\newcommand{\bX}{{\bar{X}}}
\newcommand{\bq}{\bar{q}}
\newcommand{\bomega}{\bar{\omega}}
\newcommand{\hv}{\hat{v}}
\newcommand{\dv}{\dot{v}}
\newcommand{\ha}{\hat{a}}
\newcommand{\hxi}{\hat{\xi}}
\newcommand{\hlambda}{\hat{\lambda}}
\newcommand{\txi}{\tilde{\xi}}
\newcommand{\x}{{\xi}_1} %%%%%%%
\newcommand{\tx}{\tilde{x}}
\newcommand{\ty}{\tilde{y}}
\newcommand{\td}{\tilde}
\newcommand{\pd}{{\partial}}
\newcommand{\tw}{\tilde{w}}
\newcommand{\bw}{\bar{w}}
\newcommand{\dw}{\dot{w}}
\newcommand{\hL}{\hat{L}}
\newcommand{\eps}{\epsilon}
\newcommand{\dps}{\displaystyle}
\newcommand{\Rp}{\hbox{Re}}
\newcommand{\Ip}{\hbox{Im}}
\newcommand{\grad}{\nabla}

\newcommand\be{\begin{equation}}
\newcommand\ee{\end{equation}}
\newcommand\ba{\begin{eqnarray}}
\newcommand\ea{\end{eqnarray}}
\newcommand\ds{\displaystyle}
\newcommand\nn{\nonumber}
\newcommand\bea{\begin{array}}
\newcommand\ena{\end{array}}

\section{Introduction}\label{intro}

Nonclassical viscous conservation laws
arising in multiphase fluid and solid mechanics
exhibit a rich variety of traveling wave phenomena,
including homoclinic (pulse-type) and periodic solutions
along with the standard heteroclinic (shock, or front-type)
solutions.
Here, we investigate stability of periodic traveling waves:
specifically, sufficient conditions for stability of the wave.  
Our main result is to establish $L^p$ bounds   
on the solution operator for the linearized evolution
equations, provided that there holds an appropriate
spectral condition on the linearized 
operator about the wave. 
%in this case equivalent to a 
%spectral stability criterion studied in \cite{OZ3}.
An immediate consequence is that, under mild nondegeneracy
assumptions motivated by the low-frequency spectral analysis
of \cite{OZ3}, {\it strong spectral stability
in the sense of Schneider \cite{S1,S2,S3} 
implies linearized and nonlinear
$L^1 \cap H^s\to L^p\cap H^s$ asymptotic stability}, for all $p \ge 2$
and dimensions $d\ge 3$.

The one-dimensional study on spectral stability of spatially periodic 
traveling waves of systems of viscous conservation laws
was carried out by Oh \& Zumbrun \cite{OZ1} 
in the ``quasi-Hamiltonian'' case that the traveling-wave equation
possesses an integral of motion, and in the general case
by Serre \cite{Se1}. 
An important contribution of Serre was to point out a 
larger connection between the linearized dispersion relation 
(the function $\lambda(\xi)$ relating spectra to wave number of the linearized
operator about the wave) near zero
and the homogenized system obtained by slow modulation, or WKB, approximation,
from which the various stability results of \cite{OZ1}, \cite{Se1}
may then be deduced.
In \cite {OZ3}, we extended this important observation of Serre, 
relating the linearized dispersion relation near zero to a multi-dimensional
version of the homogenized system. 
As an immediate corollary, similarly as in \cite{OZ1}, \cite{Se1} in the
one-dimensional case, this yielded as a necessary condition for
multi-dimensional stability the hyperbolicity of the multi-dimensional
homogenized system.

As noted in \cite{OZ3}, 
this relation is also the first step in the study of stability
and asymptotic behavior. 
In this article, we use the description of low-frequency
spectrum carried out in \cite{OZ3}
to obtain nonlinear stability and asymptotic behavior in dimensions
$d\ge 3$ by a modification of the Bloch decomposition
arguments introduced by Schneider in \cite{S1,S2,S3};
see Theorems \ref{t:linstab}, \ref{t:nonlinstab}, and \ref{t:behavior}.
Here, the main new difficulty is
the fact that spectra $\lambda(\xi)$ bifurcating from the
translational zero eigenvalue at $\xi=0$ are not smooth at
the origin, a standard feature of hyperbolic--parabolic
systems in multi-dimensions \cite{ZS,Z1}.
%NOTE we are able to obtain also asymptotic behavior
%in terms of a viscous regularization of the linearized WKB approximation.
%We hope to study asymptotic behavior, and especially the relation
%to behavior of the WKB approximation, in a future work.
This is not only a technical issue, but reflects
quite different behaviors in the present vs. previously considered cases.
Whereas asymptotic behavior in \cite{S1,S2,S3} was purely diffusive,
corresponding to a Gaussian kernel, 
the behavior here is convective--diffusive, 
corresponding asymptotically to a convection--diffusion wave 
in the sense of \cite{HoZ1}.
%as discussed in \cite{HoZ1} in the constant-coefficient case.
Stability and behavior in dimensions one and two
remain interesting open problems.

\section{Preliminaries}\label{prelim}
Consider a system of conservation laws
\be
u_t + \sum_j f^j(u)_{x_j} = \sum_{j, \; k} \left( B^{jk}(u)u_{x_k} \right)_{x_j},
\label{e:cons}
\ee
$u \in {\cal U} (\hbox{open}) \in \R^n$,  $f^j \in \R^n$, 
$B^{jk} \in \R^{n \times n}$, $x \in \R^d$, 
%NO, deleted:
%$d \ge 2$, 
and a periodic traveling wave solution 
\be
u=\bar{u}(x\cdot \nu -st),
\ee
of period $X$, satisfying the traveling-wave 
ordinary differential equation (ODE)
\be
( \sum_{j,k}  \nu_j \nu_k B^{jk}(\bar u) \bar u' )'= 
( \sum_j \nu_j f^j(\bar u) )'-s\bar u'
\label{e:second_order}
\ee
with initial conditions
$$
\bar u(0) = \bar u(X)=:u_0.	
$$
Integrating (\ref{e:second_order}), we reduce to a first-order profile 
equation
%Such traveling waves of velocity $s$ and direction $\nu$ 
%are associated to periodic solutions of the first-order profile equation 
\be
\sum_{j,k}  \nu_j \nu_k B^{jk}(\bar u) \bar u'= 
\sum_j \nu_j f^j(\bar u) -s \bar u -q
\label{e:profile}
\ee
encoding the conservative structure of the equations,
where $q$ is a constant of motion.

We here briefly review the generic assumptions made in \cite{OZ3,Se1}.
% assumptions ensuring that the set 
%of periodic traveling waves is a manifold of maximal dimension
%subject to the conservative properties of the equations (encoded
%in form (\ref{e:profile})).
%property of conservation of mass with conservation property.
Given 
$$
(a, s, \nu, q) \in {\cal U} 
\times \R \times S^{d-1} \times \R^n, 
$$
(\ref{e:profile}) admits a unique local
solution $u(y; a, s, \nu, q)$ such that $u(0; a, s, \nu, q)=a$. 
Denote by $X$ the period, $\omega:=1/X$ the frequency
and $M$ and $F^j$ the averages over the period:
$$
M :=\frac{1}{X} \int_0^X u(y) dy, \quad
F^j:= 
\frac{1}{X} \int_0^X 
\Big(f^j(u)- 
\sum_{k=1}^d B^{jk}(u)\omega \nu_k \partial_y u\Big) dy	
$$
when $u$ is a periodic solution of (\ref{e:profile}). Since 
these quantities are translation invariant, we consider 
the set $P$ of periodic functions $u$ that are solutions
of (\ref{e:profile}) for some triple $(s, \nu,q)$, and construct
the quotient set $\setP:=P/\cal R$ under the relation
$$
(u \; {\cal R} \; v) \Longleftrightarrow ( \exists h \in \R;\; v=u(\cdot-h)).
$$
We thus have class functions:
\begin{equation}\label{classdef}
X=X(\dot{u}),\; \omega=\Omega(\dot{u}),\;
s=S(\dot{u}),\; 
\nu=N(\dot u) , \; 
q=Q(\dot{u}),\; M=M(\dot{u}), \; F^j=F^j(\dot u),
\end{equation}
where $\dot{u}$ is the equivalence class of translates of different periodic 
functions.
Note that $\bu$ is a nonconstant periodic solution. Without loss of generality,
assume $S(\bu)=0$ and $N(\bu)=e_1$, 
%
%We now have
%Then, using the conservative form of equations (\ref{e:cons}), we may integrate
%in $x_1$ to obtain
so that (\ref{e:profile}) takes the form
$$
B^{11}(\bu)\bu'=f^1(\bu)-\bar{q}
$$
for $\bar{q}=Q(\bu)$.
Letting
$\bar{X}=X(\bu)$ and $\bar{a}=\bu(0)=u_0$, the map 
$$
(y, a, s, \nu, q) \mapsto u(y; a,s, \nu,q)-a
$$
is smooth and well-defined in a neighborhood
of $(\bar{X};\bar{a},0, e_1, \bar{q})$, and it vanishes at this special point.
Here and elsewhere, $e_j$ denotes the $j$th standard Euclidean basis element.
We assume: 

(H0) $f^j$, $B^{jk}\in C^k$, $k\ge [d/2]+1$.

(H1) $ \hbox{Re} \; \sigma(\sum_{jk} \nu_j \nu_k B^{jk})\ge \theta>0$
for all $\nu \in S^{d-1}$. 

(H2) The map 
$H: \, 
\R \times {\cal U} \times \R \times S^{d-1} \times \R^n  \rightarrow \R^n$	
taking 
$(X; a, s, \nu, q)  \mapsto u(X; a, s, \nu, q)-a$
is a submersion at point $(\bar{X}; \bar{a}, 0, e_1, \bar{q})$.
\section{The Evans function}\label{evans}

Without loss of generality taking $S(\bu)=0$, $N(\bu)=e_1$,
$\bar u=\bar{u}(x_1)$ represents a stationary solution. 
Linearizing (\ref{e:cons}) about 
$\bar{u}(\cdot)$, we obtain
\be
v_t = Lv := \sum(B^{jk}v_{x_k})_{x_j}-\sum(A^j v)_{x_j},
\label{e:lin}
\ee
where coefficients
\be
B^{jk} := B^{jk}(\bu), \;\;\;
A^jv:= Df^j(\bu)v-(DB^{j1}(\bu)v)\bu_{x_1}
\label{e:coeff}
\ee
are now periodic functions of $x_1$. 

Taking the Fourier transform in the transverse coordinate $\td{x}=
(x_2, \cdots, x_d)$, we obtain
\ba
\hv_t = L_{\td{\xi}}\hv 
& = &(B^{11} \hv_{x_1})_{x_1} -(A^1 \hv)_{x_1}
	+i (\sum_{ j\ne 1} B^{j1} \xi_j) \hv_{x_1} 	\nn \\ 
& + & i(\sum_{k\ne 1}B^{1k} \xi_k \hv)_{x_1} 
	- i \sum_{j\ne 1}A^j \xi_j \hv 
	- \sum_{j\ne 1, k\ne 1} B^{jk}\xi_k \xi_j \hv,
\label{e:fourier}
\ea
where $\td{\xi}=(\xi_2, \cdots, \xi_d)$ is the transverse frequency
vector.
The Laplace transform in time $t$ leads us to study the family of
eigenvalue equations
\ba
0=(L_{\txi}- \lambda) w
&=& (B^{11}w')'-(A^1 w)'+i \sum_{ j\ne 1}B^{j1} \xi_j w'
	+i(\sum_{k \ne 1} B^{1k} \xi_k w)'	\nn \\
&-& i\sum_{j \ne 1}A^j \xi_j w - \sum_{j\ne 1, k\ne 1}B^{jk}
	\xi_k \xi_j w - \lambda w,
\label{e:laplace}
\ea
associated with operators $L_{\txi}$ and frequency $\lambda \in \C$,
where `$'$' denotes $\partial/\partial x_1$.
Clearly, a necessary condition for stability of (\ref{e:cons}) is
that (\ref{e:laplace}) have no $L^2$ solutions $w$ 
for $\txi \in \R^{d-1}$ and Re $\lambda >0$. 
For solutions of
(\ref{e:laplace}) correspond to normal modes $v(x,t)
=\e^{\lambda t} \e^{i \txi \cdot \tx} w(x_1)$ of (\ref{e:lin}).

Multidimensional stability concerns primarily
the behavior of the perturbation of the top eigenvalue
$\lambda=0$ under small perturbations in $\txi$.
To study this behavior, we use Floquet's theory and an Evans function 
that not only depends on $\lambda$ but also on $\xi_1$ (which corresponds
to the phase shift) and $\txi$ \cite{G,OZ3}.
To define the Evans function, we choose a basis 
$\{w^1(x_1, \txi, \lambda), \ldots,
w^{2n}(x_1, \txi, \lambda) \}$ of the kernel of $L_{\txi}-\lambda$, which is
analytic in $(\txi, \lambda)$ and is real when $\lambda$ is real, for details
see \cite{OZ1,Se1}. 
Now we can define the Evans function by
\be
D(\lambda, \xi_1, \txi) := \left | 
\begin{array}{c}
  w^l(X, \txi, \lambda)-\e^{iX\xi_1}w^l(0, \txi, \lambda)  \\
  (w^l) '(X, \txi, \lambda)-\e^{iX\xi_1}(w^l) '(0, \txi, \lambda) \\
\end{array}
\right |_{1 \le l \le 2n}
\label{e:evans}
\ee
where $\xi_1 \in \R$. 
%Note that $X\xi_1$ is exactly $\theta$ in \cite{Se1}. 
We remark that $D$ is analytic everywhere, with associated
analytic eigenfunction $w^l$ for $1 \le l \le 2n$. 
A point $\lambda$ is in the spectrum
of $L_{\txi}$ if and only if  $D(\lambda, \xi)=0$ with $\xi=(\xi_1, \txi)$.

\begin{example}
In the constant-coefficient case, 
$
D(\lambda, \xi)= \prod_{l=1}^{2n} (\e^{\mu_l(\lambda, \txi)X}-\e^{i \xi_1 X}),
$
where $\mu_l, l=1, \ldots, 2n$, denote the roots of the characteristic equation
\ba
&& \left( \mu^2 B^{11}+ \mu(-A^1 +i \sum_{ j\ne 1}B^{j1} \xi_j 
	+i\sum_{k \ne 1} B^{1k} \xi_k ) \right.	\nn \\
&-& \left. ( i\sum_{j \ne 1}A^j \xi_j  + \sum_{j\ne 1, k\ne 1}B^{jk}
	\xi_k \xi_j  + \lambda I ) \right) \bar{w} =0,
\label{e:const_char}
\ea
with $w=\e^{\mu x_1} \bar{w}$. The zero set of $D$ consists of all 
$\lambda$ and $\xi_1$
such that 
$
\mu_l(\lambda, \txi)=i \xi_1 (\hbox{mod} 2 \pi i /X)
$
for some $l$. 
Setting $\mu= i \xi_1$ in (\ref{e:const_char}), 
we obtain the dispersion relation $(- B^{\xi} - i A^{\xi} - \lambda I)=0$,
%\be
%\hbox{det} (- B^{\xi} - i A^{\xi} - \lambda I)=0
%\label{e:const_dispersion}
%\ee
where $A^{\xi}= \sum_j A^j \xi_j$ and $B^{\xi}= \sum_{j, k} B^{jk} \xi_k \xi_j$.
\end{example}

%DISTRACTING, removed:
%\begin{remark}\label{ccrm}
%In the constant coefficient case, (\ref{e:const_dispersion}) yields expansions
%\be
%\lambda_j(\xi)=0- i a_j(\xi)  + o(|\xi|), \;\;j=1, \ldots, n,
%\label{e:const_lambda}
%\ee
%for the roots bifurcating from $\lambda(0)=0$, where $a_j$ denote the
%eigenvalues of $A^{\xi}$. Thus we obtain the necessary stability condition of 
%hyperbolicity, $\sigma(A^{\xi})$ real. 
%Note further that $a_j$ are not linear at the origin, but only first-order
%homogeneous; thus, $\lambda_j$ is nonsmooth at the origin,
%exhibiting a conical singularity.
%\end{remark}

%\section{The results from \cite{OZ3}} 
%\section{Low-frequency spectral behavior and the WKB expansion}\label{spectral}
\section{WKB expansion and the low-frequency limit}\label{spectral}
We now recall the results of \cite{OZ3} describing
low-frequency spectral behavior. 
As a consequence of (H0), (H2), there
is a smooth $n+d$ dimensional manifold $\setP$ of periodic solutions
$\dot u$ in the vicinity of $\bar u$, where $d$ is the spatial dimension.
On this set, one may obtain, 
rescaling by $(x,t)\to (\epsilon x, \epsilon t)$
and carrying out a formal WKB expansion as $\epsilon \to 0$ 
a closed system of $n+d$ averaged, or homogenized, equations
\ba
\label{e:wkb}
\nn
\partial_t M(\dot u) + \sum_j \partial_{x_j}(F^j(\dot u)) &=&0,  \nn \\
%\partial_t (\Omega N(\dot u)) - \nabla_x  (\Omega S(\dot u))&=&0
\partial_t (\Omega N(\dot u)) + \nabla_x  (\Omega S(\dot u))&=&0
%-M
\ea
in the $(n+d)$-dimensional unknown $\dot u$,
expected to correspond to large time-space behavior,
with an additional constraint
\begin{equation}\label{con}
\curl (\Omega N)\equiv 0
\end{equation} 
coming from the derivation of the formal expansion:
specifically, the assumption that $\Omega N$
represent the gradient $\nabla_x \phi$ of a certain phase function $\phi(x,t)$.
Here, $\Omega$, $M$, etc. are defined as in (\ref{classdef});
see \cite{OZ3} for details.

%The problem of stability of $\bar u$ can be obtained either from  
The long-time behavior of perturbations of
$\bar u$ can thus be studied formally by considering
the linearized equations of (\ref{e:wkb}) about the constant solution
$\dot u(x,t)\equiv u^0$, $u^0\sim \bar u$,
yielding the homogeneous degree $n+d$ linearized dispersion relation
\be\label{e:hatDelta}
\hat\Delta(\xi, \lambda):=
\det \left( 
\lambda \frac{\pd (M, \Omega N)}{\pd \dot{u}}(\dot{\bu})+
%TODO: note, X \omega=1 by definition!
%\sum_{j } i \bar{\omega} \bX \xi_j 
\sum_{j } i \xi_j 
\frac{\pd (F^j, 
S\Omega e_j)}{\pd \dot{u}}(\dot{\bu})  
\right )
=0.
\ee
Alternatively, it may be studied rigorously through low-frequency
expansion of the {Evans function} $D(\xi, \lambda)$, $\xi\in \R^d$,
$\lambda \in \C$,  which yields after some standard but somewhat
involved manipulations
\be
\label{e:Dexpansion}
D(\xi, \lambda)= \Delta_1(\xi, \lambda) + \Or(|\xi,\lambda|^{n+2}),
\ee
where $\Delta_1$ is a homogeneous degree $n+1$ polynomial
expressed as the determinant of a rather complicated $2n\times 2n$ matrix
in $(\xi$, $\lambda)$. 
The zero set $(\xi, \lambda(\xi))$ of $\Delta_1$ thus determines
the linearized dispersion relation for (\ref{e:cons}), with $\lambda(\xi)$
running over the tangent cone at $\xi=0$ to the surface of
low-frequency spectrum of $L$ as $\xi$ runs over $\R^d$.

Our main result in \cite{OZ3} was the following proposition
relating these two expansions,
generalizing the result of \cite{Se1} in the one-dimensional case.
Define
\be\label{e:Delta}
\Delta(\xi,\lambda):=\lambda^{1-d}\hat\Delta(\xi,\lambda),
\ee
where $\hat \Delta$ is defined as in (\ref{e:hatDelta}).

\begin{proposition} \cite{OZ3} \label{main}
Under assumptions (H0)--(H3),  $\Delta_1=\Gamma_0 \Delta$, i.e.,
\be
\label{e:tangent}
D(\xi, \lambda)= \Gamma_0 \Delta(\xi, \lambda)
+ \Or(|\xi, \lambda|^{n+2})
\ee
$\Gamma_0\ne0$ constant,
for $|\xi, \lambda|$ sufficiently small.
\end{proposition}

That is, up to an additional factor of $\lambda^{d-1}$ 
(corresponding to spurious modes not satisfying constraint (\ref{con};
see \cite{OZ3} for further discussion)
%(because of the constraint, $\curl (\Omega N)\equiv 0$, see \cite{OZ3} 
%for details),
the dispersion relation (\ref{e:hatDelta}) for the averaged system
(\ref{e:wkb}) indeed describes
the low-frequency limit of the exact linearized dispersion relation 
$ D(\xi,\lambda)=0.  $

%Roughly speaking, Theorem \ref{main} 
%states that {if} perturbed periodic waves
%exhibit coherent behavior near the unperturbed wave $\bar u$, then
%this behavior is well-described by the constrained averaged equations
%(\ref{e:wkb}).
%As an immediate consequence of Theorem \ref{main}, we obtained
%the following two corollaries, yielding
%a {necessary} condition for low-frequency
%multi-dimensional spectral stability
%strengthening the one-dimensional version 
%obtained in \cite{OZ1}, \cite{Se1}.

\begin{corollary} \cite{OZ3} \label{c:surfaces}
Assuming (H0)--(H3) and the nondegeneracy condition
\be\label{e:nondeg}
\det \left (\frac{\pd(M, \Omega N)}{\pd \dot{u}}(\dot{\bu}) \right ) \ne0,
\ee
then for $\lambda, \xi$ sufficiently small,
the zero-set of $D(\cdot, \cdot)$, corresponding to spectra of $L$, 
consists of $n+1$ characteristic surfaces:
\be
\label{e:surfaces}
\lambda_j(\xi)=-i a_j(\xi)   + o(\xi), \;\; j=1, \ldots, n+1,
\ee
where $a_j(\xi)$ denote the homogeneous, degree one eigenvalues of 
\be\label{e:Cala}
\CalA(\xi):= \sum_{j } \xi_j \frac{\pd (F^j, S\Omega e_j)} {\pd(M, \Omega N)},
\ee
excluding $(d-1)$ identically zero eigenvalues associated
with modes not satisfying (\ref{con}).
\end{corollary}

\begin{corollary} \cite{OZ3} \label{stabcondition}
Assuming (H0)--(H3) and the nondegeneracy condition (\ref{e:nondeg}),
a necessary condition for low-frequency spectral stability of
$\bar u$, defined as
Re $\lambda \le 0$ for $D(\xi,\lambda)=0$, $\xi\in \R^d$, and
$|\xi,\lambda|$ sufficiently small, 
is that the averaged system (\ref{e:wkb}) be ``weakly hyperbolic''
in the sense that the eigenvalues of $ \CalA(\xi) $ are real for
all $\xi\in \R^d$.
\end{corollary}

A consequence of Corollary \ref{c:surfaces} is that 
 $\lambda_j(\xi)$ are differentiable in $|\xi|$ at the origin for fixed 
angle $\hat \xi$,
but in general have a conical singularity in $\xi$ at $\xi=0$,
since the eigenvalues $a_j(\xi)$ of first-order system
(\ref{e:wkb}) are homogeneous degree one but typically not linear.
The low-frequency expansion of Corollary \ref{c:surfaces}
substitutes in our analysis
for the usual spectral perturbation analysis by formal
series expansion/Fredholm alternative in the standard
case (as in \cite{S1,S2,S3}) that $\lambda_j(\xi)$ vary smoothly in $\xi$.

\begin{remark}\label{ver}
\textup{
The low-frequency stability condition of Corollary \ref{stabcondition} 
has been verified numerically for the example of 
isentropic van der Waals gas dynamics \cite{O}, for which periodic
solutions are known to appear.
On the other hand, it was shown in \cite{OZ1} that high-frequency
instabilities appear for these waves, so they are not in the end
stable.
}
\end{remark}

\section{Bloch--Fourier decomposition 
and the spectral stability conditions}\label{bloch}

Following \cite{G,S1,S2,S3}, we define the family of operators
\be
L_{\xi} = \e^{-i \xi_1 x_1} L_{\txi}  \e^{i \xi_1 x_1}
\label{e:part}
\ee
operating on the class of $L^2$ periodic functions on $[0,X]$;
the $(L^2)$ spectrum
of $L_{\txi}$ is equal to the union of the
spectra of all $L_{\xi}$ with $\xi_1$ real with associated 
eigenfunctions
\be
w(x_1, \txi,\lambda) := \e^{i \xi_1 x_1} q(x_1, \x, \txi, \lambda),
\label{e:efunction}
\ee
where $q$, periodic, is an eigenfunction of $L_{\xi}$.
By continuity of spectrum, 
%(it may not be analytic [Kat]), 
and discreteness of the spectrum of the elliptic operators $L_\xi$ on
the compact domain $[0,X]$,
we have that the spectra of $L_{\xi}$
may be described as the union of countably many continuous
surfaces $\lambda_j(\xi)$. 

The spectrum of each $L_{\xi}$ may alternatively be
characterized as the zero set for fixed $\xi$ of the 
periodic Evans function $D(\x, \txi, \lambda)$;
see \cite{G,OZ3}.
%NOT NEEDED, removed:
%Likewise the spectrum of $L_{\txi}$ may be described as the set of
%all $\lambda$ such that $D(\x, \txi, \lambda)$ vanishes
%for some real $\x$; see \cite{G,OZ3}. 
More, a fundamental result of Gardner \cite{G} is that {\it
the order of vanishing of the Evans function in $\lambda$ is 
equal to the multiplicity of $\lambda$ as an eigenvalue of $L_\xi$}.
Thus, we have a description of the eigenstructure of $L_\xi$
through Corollary \ref{c:surfaces} for $\lambda$, $\xi$ sufficiently
small in terms of the characteristics of the first-order
hyperbolic system (\ref{e:wkb}).

Without loss of generality taking $X=1$,
recall now the {\it Bloch--Fourier representation}
\be\label{Bloch}
u(x)=
\Big(\frac{1}{2\pi }\Big)^d \int_{-\pi}^{\pi}\int_{\R^{d-1}}
e^{i\xi\cdot x}\hat u(\xi, x_1)
d\xi_1\, d\tilde \xi
\ee
of an $L^2$ function $u$, where 
$\hat u(\xi, x_1):=\sum_k e^{2\pi ikx_1}\hat u(\xi_1+ 2\pi k,\tilde \xi)$
are periodic functions of period $X=1$, $\hat u(\tilde \xi)$
denoting with slight abuse of notation the Fourier transform of $u$
in the full variable $x$.  
By Parseval's identity, the Bloch--Fourier transform
$u(x)\to \hat u(\xi, x_1)$ is an isometry in $L^2$:
\be\label{iso}
\|u\|_{L^2(x)}= 
\|\hat u\|_{L^2(\xi; L^2(x_1))},
\ee
where $L^2(x_1)$ is taken on $[0,1]$ and $L^2(\xi)$
on $[-\pi,\pi]\times \R^{d-1}$.
Moreover, it diagonalizes the periodic-coefficient operator $L$,
yielding the {\it inverse Bloch--Fourier transform representation}
\be\label{IBFT}
e^{Lt}u_0=
\Big(\frac{1}{2\pi }\Big)^d \int_{-\pi}^{\pi}\int_{\R^{d-1}}
e^{i\xi \cdot x}e^{L_\xi t}\hat u_0(\xi, x_1)
d\xi_1\, d\tilde \xi
\ee
relating behavior of the linearized system to 
that of the diagonal operators $L_\xi$.

Along with the assumptions (H0)--(H2), we assume the
{\it strong spectral stability} conditions: 
 
%OLDER... (equivalent but less natural) 
%(D2) $\hbox{\rm Re} \lambda_j(\xi) < -\theta |\xi|^2$, $\theta \ge 0$,
%for all real $\xi$ with $|\xi|$ sufficiently small,
%$\lambda_j$ as defined in Corollary \ref{c:surfaces};
%equivalently,

(D1) $\sigma(L) \subset \{ \hbox{\rm Re} \lambda <0 \} \cup \{ 0\}$.

(D2) $\hbox{\rm Re} \sigma(L_{\xi}) \le -\theta |\xi|^2$, $\theta>0$, 
for $\xi\in \R^d$ and $|\xi|$ sufficiently small.

%Comments from Kevin: About analyticity along radii:
%Recall in Oh-Z.2 the lemma 1.1, which says that curves under
%D1-D3 are at least C2 functions.  I NOW THINK THAT THIS
%result, and not full analyticity, WAS PROBABLY ALL WE NEEDED
%to get pointwise bounds (meaning modulus bounds, not heat kernel
%plus t^{-1/2} faster decaying error....).  And, this certainly DOES
%hold also in the multi-d case, by same argument, at least along each
%radial direction.  So, a nice thing to do might be to work from this
%assumption alone and forget about analyticity (I THINK IT WOULD
%IN FACT BE OK FOR MOST PURPOSES).

By Corollary \ref{c:surfaces}, either of (D1), (D2) implies that
the eigenvalues $a_j(\xi)$ of (\ref{e:Cala}) are real.
We make the further nondegeneracy hypothesis:

(H3) The values $a_j(\xi)$ are distinct.

\noindent 
Conditions (D1)--(D2) are exactly the spectral assumptions of
\cite{S1,S2,S3}, corresponding to ``dissipativity'' of the
large-time behavior of the linearized system.
Condition (H3) corresponds to strict hyperbolicity of the 
averaged system (\ref{e:wkb});
presumably it could be removed with further effort/alternative hypotheses. 

\section{Linearized stability estimates}\label{linests}

By standard spectral perturbation theory \cite{K}, the total
eigenprojection $P(\xi)$ onto the eigenspace of $L_\xi$ 
associated with the eigenvalues $\lambda_j(\xi)$, $j=1,\dots, n+1$ 
described in Corollary \ref{c:surfaces}
is well-defined and analytic in $\xi$ for $\xi$ sufficiently small,
since these (by discreteness of the spectra of $L_\xi$) are
separated at $\xi=0$ from the rest of the spectrum of $L_0$.
Introducing a smooth cutoff function $\phi(\xi)$ that
is identically one for $|\xi|\le \eps$ and identically
zero for $|\xi|\ge 2\eps$, $\eps>0$ sufficiently small,
we split the solution operator $S(t):=e^{Lt}$ into
low- and high-frequency parts
\be\label{SI}
S^I(t)u_0:=
\Big(\frac{1}{2\pi }\Big)^d \int_{-\pi}^{\pi}\int_{\R^{d-1}}
e^{i\xi \cdot x}
\phi(\xi)P(\xi) e^{L_\xi t}\hat u_0(\xi, x_1)
d\xi_1\, d\tilde \xi
\ee
and
\be\label{SII}
S^{II}(t)u_0:=
\Big(\frac{1}{2\pi }\Big)^d \int_{-\pi}^{\pi}\int_{\R^{d-1}}
e^{i\xi \cdot x}
\big(I-\phi P(\xi)\big)
e^{L_\xi t}\hat u_0(\xi, x_1)
d\xi_1\, d\tilde \xi.
\ee

\subsection{High-frequency bounds}\label{HF}
By standard sectorial bounds \cite{He,Pa} and spectral separation
of $\lambda_j(\xi)$ from the remaining spectra of $L_\xi$,
we have trivially the exponential decay bounds
\ba\label{semigp}
\|e^{L_\xi}(I-\phi P(\xi))f\|_{L^2([0,X])}
&\le & Ce^{-\theta t}\|f\|_{L^2([0,X])},\\
\|e^{L_\xi}(I-\phi P(\xi))\partial_x f\|_{L^2([0,X])}
&\le & Ct^{-\frac{1}{2}}e^{-\theta t}\|f\|_{L^2([0,X])},\\
\|\partial_x e^{L_\xi}(I-\phi P(\xi)) f\|_{L^2([0,X])}
&\le & Ct^{-\frac{1}{2}}e^{-\theta t}\|f\|_{L^2([0,X])},
\ea
for $\theta$, $C>0$.
Together with (\ref{iso}), these give immediately the
following estimates.

\begin{proposition}\label{p:hf}
Under assumptions (H0)--(H3) and (D1)--(D2),
for some $\theta$, $C>0$,
and all $t>0$, $2\le p\le \infty$,
\ba\label{SIIest}
\|S^{II}(t)f\|_{L^2(x)}&\le& Ce^{-\theta t}\|f\|_{L^2(x)},\\
\|\partial_x S^{II}(t) f\|_{L^2(x)},\;
\label{SII2}
\|S^{II}(t)\partial_x f\|_{L^2(x)}&\le& 
Ct^{-\frac{1}{2}}e^{-\theta t}\|f\|_{L^2(x))},\\
\label{SII3}
\|S^{II}(t) f\|_{L^p(x)}&\le& 
Ct^{-\frac{d}{2}(\frac{1}{2}-\frac{1}{p})}e^{-\theta t}\|f\|_{L^2(x)}.
\ea
\end{proposition}

\begin{proof}
The first two inequalities follow immediately by (\ref{iso}).
The third follows for $p=\infty$ by Sobolev embedding from
$$
\|S^{II}(t) f\|_{L^p(\tilde x; L^2(x_1))}\le 
Ct^{-\frac{d-1}{2}(\frac{1}{2}-\frac{1}{p})}e^{-\theta t}\|f\|_{L^2([0,X])}
$$
and
$$
\|\partial_{x_1} S^{II}(t) f\|_{L^p(\tilde x; L^2(x_1))}\le 
Ct^{-\frac{d-1}{2}(\frac{1}{2}-\frac{1}{p})- \frac{2}{2}
}e^{-\theta t}\|f\|_{L^2([0,X])},
$$
which follow by an application of (\ref{iso}) in the $x_1$ variable and 
the Hausdorff--Young inequality $\|f\|_{L^\infty(\tilde x)}\le 
\|\hat f\|_{L^1(\tilde \xi)}$ in the variable $\tilde x$.
The result for general $2\le p\le \infty$ then follows by
$L^p$ interpolation.
\end{proof}

\subsection{Low-frequency bounds}\label{LF}
Denote by 
\be\label{GI}
G^I(x,t;y):=S^I(t)\delta_y(x)
\ee
the Green kernel associated with $S^I$, and
\be\label{GIxi}
[G^I_\xi(x_1,t;y_1)]:=\phi(\xi)P(\xi) e^{L_\xi t}[\delta_{y_1}(x_1)]
\ee
the corresponding kernel appearing within the Bloch--Fourier representation
of $G^I$, where the brackets on $[G_\xi]$ and $[\delta_y]$
denote the periodic extensions of these functions onto the whole line.
Then, we have the following descriptions of $G^I$, $[G^I_\xi]$,
deriving from the Evans function analysis of Corollary \ref{c:surfaces}.

\begin{proposition}\label{kernels}
Under assumptions (H0)--(H3) and (D1)--(D2),
\ba\label{Gxi}
[G^I_\xi(x_1,t;y_1)]&=& \phi(\xi)\sum_{j=1}^{n+1}e^{\lambda_j(\xi)t}
q_j(\xi,x_1)\tilde q_j(\xi, y_1)^*,\\
\label{Gxi2}
G^I(x,t;y)&=& 
\Big(\frac{1}{2\pi }\Big)^d \int_{\R^{d}} e^{i\xi \cdot (x-y)} 
[G^I_\xi(x_1,t;y_1)] d\xi \\
\nonumber
&=&
\Big(\frac{1}{2\pi }\Big)^d \int_{\R^{d}}
e^{i\xi \cdot (x-y)}
\phi(\xi)
\sum_{j=1}^{n+1}e^{\lambda_j(\xi)t} q_j(\xi,x_1)\tilde q_j(\xi, y_1)^*
d\xi,
\ea
where $*$ denotes matrix adjoint, or complex conjugate transpose,
$q_j(\xi,\cdot)$ and $\tilde q_j(\xi,\cdot)$ 
are right and left eigenfunctions of $L_\xi$ associated with eigenvalues
$\lambda_j(\xi)$ defined in Corollary \ref{c:surfaces},
normalized so that $\langle \tilde q_j,q_j\rangle\equiv 1$, where
$\lambda_j/|\xi|$ is a smooth function of $|\xi|$ and $\hat \xi:=\xi/|\xi|$
and $q_j$ and $\tilde q_j$ are smooth functions of
$|\xi|$, $\hat \xi:=\xi/|\xi|$, and $x_1$ or $y_1$, with 
$\Re \lambda_j(\xi)\le -\theta|\xi|^2$.
\end{proposition}

\begin{proof}
Smooth dependence of $\lambda_j$ and of $q$, $\tilde q$ as functions
in $L^2[0,X]$ follow from standard spectral perturbation theory
\cite{K} using the fact that $\lambda_j$ split to first order
in $|\xi|$ as $\xi$ is varied along rays through the origin,
and that $L_{\xi}$ varies smoothly with angle $\hat \xi$.
Smoothness of $q_j$, $\tilde q_j$ in $x_1$, $y_1$ then follow from
the fact that they satisfy the eigenvalue equation for $L_\xi$,
which has smooth, periodic coefficients.
Likewise, (\ref{Gxi}) is immediate from the spectral decomposition
of elliptic operators on finite domains.
Substituting (\ref{GI}) into (\ref{SI}) 
and computing 
\be\label{comp1}
\widehat{\delta_y}(\xi,x_1)=
\sum_k e^{2\pi i kx_1}\widehat{\delta_y}(\xi + 2\pi k e_1)=
\sum_k e^{2\pi i kx_1}e^{-i\xi \cdot y-2\pi i ky_1}
= e^{-i\xi \cdot y}[\delta_y(x)],
\ee
where the second and third equalities follow from the fact that
the Fourier transform either continuous or discrete of 
the delta-function is unity, we obtain 
\ba\label{GIsub}
G^I(x,t;y)&=&
\Big(\frac{1}{2\pi }\Big)^d \int_{-\pi}^{\pi}\int_{\R^{d-1}}
e^{i\xi \cdot x} \phi P(\xi) e^{L_\xi t} \widehat{\delta_y}(\xi,x_1)d\xi\\
\nonumber
&=&
\Big(\frac{1}{2\pi }\Big)^d \int_{-\pi}^{\pi}\int_{\R^{d-1}}
e^{i\xi \cdot (x-y)}  \phi P(\xi)e^{L_\xi t} [\delta_y(x)] d\xi,
\ea
yielding (\ref{Gxi2}) by (\ref{GIxi}) and the fact that $\phi$
is supported on $[-\pi,\pi]$.
\end{proof}

\begin{proposition} \label{Gbds}
Under assumptions (H0)-(H3) and (D1)-(D3), 
\be\label{GIest}
\sup_{y}\|G^I(\cdot, t,;y) \|_{L^p(x)},
\;
\sup_{y}\|\partial_{x,y} G^I(\cdot, t,;y) \|_{L^p(x)}
 \le  C (1+t)^{-\frac{d}{2}(1-\frac{1}{p}}
\ee
for all $2 \le p \le \infty$, $t \ge 0 $, where $C>0$ is independent of $p$.
\end{proposition}

\begin{proof}
From representation (\ref{Gxi})(ii) and $\Re \lambda_j(\xi)\le -\theta |\xi|^2$,
we obtain by the triangle inequality
\be
\|G^I\|_{L^\infty(x,y)}\le C\|e^{-\theta |\xi|^2 t} \phi(\xi)\|_{L^1(\xi)}
 \le  C (1+t)^{-\frac{d}{2}},
\ee
verifying the bounds for $p=\infty$.  Derivative bounds follow similarly,
since derivatives falling on $q_j$ or $\tilde q_j$ are harmless, whereas
derivatives falling on $e^{i\xi\cdot(x-y)}$ bring down a factor
of $\xi$, again harmless because of the cutoff function $\phi$.

To obtain bounds for $p=2$, we note that (\ref{Gxi}(ii) may be viewed
itself as a Bloch--Fourier decomposition with respect to variable
$z:=x-y$, with $y$ appearing as a parameter.
Recalling (\ref{iso}), we may thus estimate
\ba
\sup_y \|G^I(x,t;y)\|_{L^2(x)}&=&
\sum_j \sup_y \|\phi(\xi) e^{-\lambda_j(\xi)t}
q_j(\cdot, z_1)\tilde q_j^*(\cdot, y_1)\|_{L^2(\xi; L^2(z_1\in [0,X]))}\\
&\le&
C\sum_j \sup_y \|\phi(\xi) e^{-\theta |\xi|^2t} \|_{L^2(\xi)}
\|q_j\|_{L^2(0,X)} \|\tilde q_j\|_{L^\infty(0,X)}
\\
&\le&
 C (1+t)^{-\frac{d}{4}},
\ea
where we have used in a crucial way the boundedness of $\tilde q_j$;
derivative bounds follow similarly.

Finally, bounds for $2\le p\le \infty$ follow by $L^p$-interpolation.
\end{proof}

\begin{remark}
\textup{
In obtaining the key $L^2$-estimate, we have used in an essential
way the periodic structure of $q_j$, $\tilde q_j$.  For, viewing
$G^I$ as a general pseudodifferential expression rather than
a Bloch--Fourier decomposition, we find that the smoothness of
$q_j$, $\tilde q_j$ is not sufficient to apply standard $L^2\to L^2$
bounds of H\"ormander, which require blowup in $\xi$ derivatives
at {\it less than} the critical rate $|\xi|^{-1}$ found here; see,
e.g., \cite{H} for further discussion.
}
\end{remark}

\begin{remark}\label{greenformula}
\textup{
Computation (\ref{comp1}) applied to the full solution formula
(\ref{IBFT}) yields the fundamental relation
\be\label{greenform}
G(x,t;y)=
\Big(\frac{1}{2\pi }\Big)^d \int_{-\pi}^{\pi}\int_{\R^{d-1}}
e^{i\xi \cdot (x-y)}[G_\xi(x_1,t;y_1)]d\xi
\ee
which, provided $\sigma(L_\xi)$ is semisimple, yields the simple
formula
$$
G(x,t;y)=
\Big(\frac{1}{2\pi }\Big)^d \int_{-\pi}^{\pi}\int_{\R^{d-1}}
e^{i\xi \cdot (x-y)}\sum_j e^{\lambda_j(\xi)t}q_j(\xi,x_1)
\tilde q_j(\xi, y_1)^* d\xi
$$
resembling that of the constant-coefficient case,
where $\lambda_j$ runs through the spectrum of $L_\xi$.
Relation (\ref{greenform}) underlies both the present analysis
and the technically rather different approach of \cite{OZ2},
with the basic idea in both cases being to separate off the
principal part of the series involving small $\lambda_j(\xi)$
and estimate the remainder as a faster-decaying residual.
}
\end{remark}

\begin{corollary}\label{Sbd}
Under assumptions (H0)--(H3) and (D1)--(D2),
for all $p\ge 2$, $t\ge 0$,
\ba\label{SIest}
\|S^I(t)f\|_{L^p},\;
\|\partial_x S^I(t)f\|_{L^p}, \;
\|S^I(t) \partial_x f\|_{L^p}
 &\le&  C (1+t)^{-\frac{d}{2}(1-\frac{1}{p})}\|f\|_{L^1}.
\ea
\end{corollary}

\begin{proof}
Immediate, from (\ref{GIest}) and the triangle inequality,
as, for example,
$$
\|S^I(t)f(\cdot )\|_{L^p}=
\Big\|\int_{\R^d}G^I(x,t;y)f(y)dy\Big\|_{L^p(x)}
\le
\int_{\R^d}\sup_y \|G^I(\cdot ,t;y)\|_{L^p}|f(y)|dy.
$$
\end{proof}

{\bf Additional estimates.}
For general interest, we state also some easy-to-obtain generalizations,
even though they are not needed in the analysis.
From boundedness of the spectral projections $P_j(\xi)=
q_j \langle \tilde q_j, \cdot\rangle$ in $L^2[0,X]$ and their derivatives, 
another consequence of first-order splitting of eigenvalues
$\lambda_j(\xi)$ at the origin, we obtain 
boundedness of $\phi(\xi) P(\xi)e^{L_\xi t}$
and thus, by (\ref{iso}), the global bounds
\be\label{triv}
\|S^I(t)f\|_{L^2(x)},\;
\|\partial_x S^I(t)f\|_{L^2(x)},\;
\|S^I(t)\partial_x f\|_{L^2(x)}
\le  C\|f\|_{L^2(x)}
\quad \hbox{\rm for all $t\ge 0$}.
\ee
By Riesz--Thorin interpolation between (\ref{triv}) and (\ref{SIest}),
%TODO: ref?
we obtain the following, apparently sharp bounds between various
$L^q$ and $L^p$.

\begin{corollary}\label{RT}
Assuming (H0)--(H3) and (D1)--(D2),
for all $1\le q\le 2$, $p\ge 2$, $t\ge 0$,
\ba\label{RTest}
\|S^I(t)f\|_{L^p},\;
\|\partial_x S^I(t)f\|_{L^p}, \;
\|S^I(t) \partial_x f\|_{L^p}
 &\le&  C (1+t)^{-\frac{d}{2}(\frac{1}{q}-\frac{1}{p})}\|f\|_{L^q}.
\ea
\end{corollary}

\subsection{Short-time bounds}\label{short}

Finally, we recall the following short-time bounds
on the full solution operator $S(t)=e^{Lt}$, following from standard
semigroup theory \cite{Pa} for the second-order elliptic operator
$L$.

\begin{proposition}\label{p:Sbd}
Assuming (H0)--(H3) and (D1)--(D2),
for $1\le p\le \infty$, $0\le t\le 1$,
\ba\label{shortSest}
\|S(t)f\|_{L^p}
 &\le&  C \|f\|_{L^p},\\
\label{shortSderest}
\|\partial_x S(t)f\|_{L^p}, \; \|S(t) \partial_x f\|_{L^p}
 &\le&  C t^{-\frac{1}{2}}
\|f\|_{L^p}.
\ea
\end{proposition}

\subsection{Linearized stability in dimensions $d\ge1$}\label{linstability}

\begin{theorem}\label{t:linstab}
Assuming (H0)-(H3), spectral stability (D1)-(D2), implies
$L^1 \cap L^p \rightarrow L^p$ asymptotic stability of the linear equation
(\ref{e:cons}), for all $p \ge 2$ and dimensions $d \ge 1$, with 
\ba\label{Sbound}
\|S(t)u_0\|_{L^p} &\le& 
C (1+t)^{-\frac{d}{2}(1-\frac{1}{p})} \|u_0 \|_{L^1\cap L^p},\\
\label{derSbound}
\|S(t)\partial_{x}u_0\|_{L^p} &\le& 
C t^{-\frac{1}{2}} (1+t)^{\frac{1}{2}-\frac{d}{2}(1-\frac{1}{p})} \|u_0 \|_{L^1\cap L^p}
\quad \hbox{\rm for all $t \ge 0$.}
\ea
\end{theorem}

\begin{proof}  
Immediate, from (\ref{SIIest}), (\ref{SIest}), and (\ref{shortSest}).
\end{proof}

\section{Nonlinear stability in dimensions $d\ge 3$}\label{stability}

Define now the perturbation variable 
$v:=u-\bar u$ for $u$ a solution of (\ref{e:cons}).

\begin{proposition}\label{damping}
Assuming (H0)-(H3), let $v_0\in H^k$, 
$k\ge [d/2]+1$ as in (H0), and suppose that
for $0\le t\le T$, the $H^k$ norm of $v$
remains bounded by a sufficiently small constant. 
There are then constants $\theta_{1,2}>0$ so that, for all $0\leq t\leq T$,
\begin{equation}\label{Ebounds}
\|v(t)\|_{H^k}^2 \leq C \rme^{-\theta_1 t} \|v(0)\|^2_{H^k} + C \int_0^t \rme^{-\theta_2(t-s)} |v|_{L^2}^2 (s)\,ds.
\end{equation}
\end{proposition}

\begin{proof}
Subtracting the equations for $u$ and $\bar u$, we may write the
nonlinear perturbation equation as
\ba\label{vperturteq}
v_t + \sum_j (df_j(\bar u)v)_{x_j} - \sum_{j,k}(B_{jk}(u)v_{x_j})_{x_k}
&=&\sum_{j,k}((B_{jk}(\bar u+ v)-B_{jk}(\bar u))\bar u_{x_j})_{x_k}\\
\nonumber
& &-\sum_{j}(f_j(\bar u+v)-f_j(\bar u)-df_j(\bar u)v)_{x_j}.
\ea
In the uniformly elliptic case
$$
\Re \sum_{jk}B_{jk}\nu_j\nu_k \ge \theta|\nu|^2, \, \theta>0
$$
we may take the $L^2$ inner product in $x$ of $\sum_{k=0}^k\partial_x^{2k}v$ 
against (\ref{vperturteq}), integrate by parts, apply 
G\"arding's inequality, and rearrange the resulting terms,
to arrive at the inequality
\[
\partial_t \|v\|_{H^k}^2(t) \leq -\theta \|\partial_x^{k+1} v\|_{L^2}^2 +
C \|v\|_{H^k}^2, 
\]
where $\theta>0$, for some sufficiently large $C>0$, so long as $\|v\|_{H^k}$ remains bounded. Using the Sobolev interpolation
$
\|v\|_{H^k}^2 \leq \tilde{C}^{-1} \|\partial_x^{k+1} v\|_{L^2}^2 + \tilde{C} \| v\|_{L^2}^2
$
for $\tilde{C}>0$ sufficiently large, we obtain 
$
\partial_t \|v\|_{H^k}^2(t) \leq -\tilde{\theta} \|v\|_{H^k}^2 + 
C \|v\|_{L^2}^2,
$
from which (\ref{Ebounds}) follows by Gronwall's inequality.
In the general case, we perform an analogous pseudodifferential
estimate using a frequency-dependent symmetrizer to obtain the
same result.
We omit the (standard) details.
\end{proof}

\begin{theorem}\label{t:nonlinstab}
Assume (H0)-(H3). Then, spectral stability, (D1)-(D3), implies
nonlinear asymptotic stability of $\bar u$ from $L^1 \cap H^k\to L^p$,
$k\ge [d/2]+1$ as in (H0) and $p\ge 2$, for dimensions $d \ge 2$, 
with
\ba\label{e:pmain}
\|u(\cdot,t)-\bar{u} \|_{L^p} &\le& 
C  (1+t)^{-\frac{d}{2}(1-\frac{1}{p})} \|u_0-\bar u\|_{L^1\cap H^k},\\
\label{e:Hmain}
\|u(\cdot,t)-\bar{u} \|_{H^k} &\le& 
C  (1+t)^{-\frac{d}{4}} \|u_0-\bar u\|_{L^1\cap H^k}
\quad
\hbox{\rm for all $t \ge 0$}.
\ea
\end{theorem}

\begin{proof}
Taylor expanding about $\bar u$, we obtain the alternative
perturbation equation
\be
v_t-Lv=\sum_j Q^j(v, \grad v)_{x_j},
\ee
where
\ba	\label{e:q}
Q^j(v, \grad v)&=& \Or(|v|^2+|v||\grad v|)
%\\
%\partial_{x}Q^j(v, \grad v)&=& \Or(|v||\grad v|+|\grad v|^2 + |v||\partial_x^2 v|)
\ea
as long as $|v|$ remains bounded by some fixed constant. 
By Duhamel's principle, we have
\be \label{e:v}
v(\cdot ,t)=
S(t)v_0 +\int_0^t S(t-s) \sum_j \partial_{y_j}Q^j(\cdot,s)ds.
\ee

{\bf Case $L^2\cap H^k$.}
Define now
\be \label{e:eta}
%\eta(t):=\hbox{sup}_{0 \le s \le t, 2 \le p \le \infty} \|v(\cdot, s) \|_{L^p}(1+s)^{\frac{d}{2}(1-\frac{1}{p})}.
\eta(t):=\hbox{sup}_{0 \le s \le t} \|v(\cdot, s) \|_{L^p}
(1+s)^{\frac{d}{4}}.
\ee
By Proposition \ref{damping}, $\|v\|_{H^{2}}$, hence $\|v_t\|_{L^2}$,
remains small so long as $\eta$ remains small, hence $\eta$ remains
continuous so long as it remains small.
We first establish
\be \label{e:claim}
\eta(t) \le C(\eta_0+\eta(t)^2),
\quad
\eta_0:=\|v_0\|_{L^1\cap H^k},
\ee
from which it follows by continuous induction that 
$\eta(t) \le 2C \eta_0$ for $t \ge0$, if $\eta_0 < 1/ 4C$, 
yielding by (\ref{e:eta}) the result (\ref{e:pmain}) for $p=2$. 
This in turn yields (\ref{e:Hmain}) by Proposition \ref{damping}.

By Proposition \ref{damping}, (\ref{e:eta}), and Sobolev embedding,
\ba
|Q^j(v, \grad v)(\cdot,t)|_{L^1\cap L^p} \le
(|\grad v|_{L_{2}} + |v|_{L^{2}} )|v|_{L^2\cap L^\infty} 
		&\le& C \eta(t)^2 (1+t)^{-\frac{d}{2}}
\ea
%and
%\ba
%|\partial_{x_j}Q^j(v, \grad v)(\cdot,t)|_{L^1\cap L^p} \le
%|v|_{H_{2}}^2 + |v|_{H^2}|v|_{L^\infty}
%+ |\nabla v|_{H^2}^2 
%\le C \eta(t)^2 (1+t)^{-\frac{d}{2}}
%\ea
for all $2\le p\le \infty$, in particular $p=2$.
Substituting into (\ref{e:v}) and using (\ref{Sbound}),
(\ref{derSbound}), we thus obtain
\ba\label{est2}
|v(\cdot, t)|_{L^2} &\le& C\eta_0 (1+t)^{-\frac{d}{4}}	\\	\nn
			   &+& C \eta(t)^2 \int_0^t 
(1+t-s)^{\frac{1}{2}-\frac{d}{4}} (t-s)^{-\frac{1}{2}}
(1+s)^{-\frac{d}{2}}	\\  \nn
			  &\le&C(\eta_0+\eta(t)^2)
(1+t)^{\max\{-\frac{d}{4}, 1-\frac{3d}{4}\}}
\ea
so long as $(1+s)^{-\frac{d}{2}}$ is integrable, i.e.,
for $d\ge 3$.
Noting that $1-\frac{3d}{4}\le -\frac{d}{4}$ for $d\ge 2$, 
we obtain (\ref{e:claim}) as claimed.
This completes the proof of (\ref{e:pmain})--(\ref{e:Hmain}) for $p=2$.

{\bf Case $L^p$, $2\le p\le \infty$.}
Substituting again into (\ref{e:v}) and using (\ref{Sbound}),
(\ref{derSbound}), we obtain
\ba\label{estp}
|v(\cdot, t)|_{L^2} &\le& C\eta_0 (1+t)^{-\frac{d}{2}(1-\frac{1}{p})}	\\	\nn
			   &+& C \eta(t)^2 \int_0^t 
(1+t-s)^{\frac{1}{2}-\frac{d}{2}
} (t-s)^{-\frac{1}{2}(1-\frac{1}{p})}	
(1+s)^{-\frac{d}{2}}	\\  \nn
			  &\le&C_2\eta_0
(1+t)^{-\frac{1}{2}(1-\frac{1}{p}) }	
\ea
for $d\ge 3$. 
This completes the proof of (\ref{e:pmain}) for $2\le p\le \infty$.
\end{proof}

\begin{remark}
\textup{
In the critical dimension $d=2$, our stability argument barely fails,
due to the appearance of a $\log t$ term coming from the expression
$\int(1+s)^{-\frac{d}{2}}ds$.
}
\end{remark}

\section{Asymptotic behavior in dimensions $d\ge 3$}\label{behavior}

\subsection{Second-order approximation}
Expressing $\xi$ in polar coordinates $(r,\hat \xi)$, where $\xi=r\hat \xi$,
$r=|\xi|$, $\hat \xi=\xi/|\xi|$, and expanding
$$
\lambda_j(r, \hat \xi)= a_j(\hat \xi)r+ b_j(\hat \xi)r^2 + O(r^3),
$$
define now
\ba\label{lstar}
\lambda_j^*(\xi)&:=& a_j(\hat \xi)r+ b_j(\hat \xi)r^2, 
\quad j=1, \dots, n+1,\\
\label{Gxistar}
[G^\dagger_\xi(x_1,t;y_1)]&:=& \phi(\xi)\sum_{j=1}^{n+1}e^{\lambda_j^\dagger(\xi)t}
q_j((0,\hat \xi),x_1)\tilde q_j((0,\hat \xi), y_1)^*,\\
\label{Gxi2star}
G^\dagger(x,t;y)&:=& 
\Big(\frac{1}{2\pi }\Big)^d \int_{\R^{d}} e^{i\xi \cdot (x-y)} 
[G^\dagger_\xi(x_1,t;y_1)] d\xi ,
\ea
noting that $q((r,\hat \xi),x_1)$ and $\tilde q((r,\hat \xi),y_1)$ 
by Proposition \ref{kernels} are smooth in all arguments.
Define likewise
\be\label{ustar}
v^\dagger(x,t):= \int_{\R^d} G^\dagger(x,t;y)v_0(y)dy
= \int_{\R^d} G^\dagger(x,t;y)(u_0-\bar u)(y)dy.
\ee

\begin{proposition}\label{t:behavior1}
Assuming (H0)-(H3) and (D1)-(D3), 
for all $p\ge 2$, $d\ge 3$, $t \ge 0$, 
\ba\label{e:bp}
\|u(\cdot,t)-\bar{u}-v^\dagger(\cdot,t) \|_{L^p} &\le& 
C  (1+t)^{-\frac{d}{2}(1-\frac{1}{p})-\frac{1}{2}} 
\|u_0-\bar u\|_{L^1\cap H^k},\\
\label{e:bH}
\|u(\cdot,t)-\bar{u}-v^\dagger(\cdot,t) \|_{H^k} &\le& 
C  (1+t)^{-\frac{d}{4}-\frac{1}{2}} \|u_0-\bar u\|_{L^1\cap H^k}.
\ea
\end{proposition}
\begin{proof}
By smoothness of $q_j$, $\tilde q_j$ in angle $\hat \xi$,
$$
|q_j((r,\hat \xi),x_1)- q_j((0,\hat \xi),x_1)|, \;
|\tilde q_j((r,\hat \xi),x_1)- q_j\tilde ((0,\hat \xi),x_1)|
=O(|\xi|),
$$
whence we easily obtain (\ref{e:bp}) and (\ref{e:bH}) by estimates like
those used in the proof of Theorem \ref{t:nonlinstab}.
We omit the details; see, for example, \cite{HoZ1} for similar computations.
\end{proof}

\subsection{Reduction to constant-coefficients}

%The decay rates of Proposition \ref{t:behavior1} 
%compared with those of Theorem \ref{t:nonlinstab}
%suggest that the (linear) second-order approximant $v^\dagger$ is the
%principal part of the perturbation $u-\bar u$,
%with the remainder $u-\bar u-v^\dagger$ decaying faster
%than $u-\bar u$.
%In order to verify this, we first make a Floquet-type
%reduction to constant-coefficients.

Noting that ${\rm Span } \{q_j\}(0,\hat \xi)$ 
and ${\rm Span } \{\tilde q_j\}(0,\hat \xi)$ 
are independent of the angle $\hat \xi$, corresponding for each $\hat \xi$
to the right and left zero eigenspaces $\Sigma_0$
and $\tilde \Sigma_0$ of $L_0$, we may
choose a fixed real (since $L_0$ has real coefficients and
eigenvalue $0$ is real) 
pair of dual bases $\pi_j$ and $\tilde \pi_j$
of $\Sigma_0$ and $\tilde \Sigma_0$, $\langle \tilde \pi_j, \pi_k\rangle=
\delta_j^k$, $j=1,\dots, n+1$, to obtain
\ba\label{base}
q_j((0,\hat \xi),x_1)&=&\sum_k \alpha_{kj}(0,\hat \xi)\pi_k(x_1),\\
\nonumber
\tilde q_j((0,\hat \xi),y_1)&=&
\sum_k \tilde \alpha_{kj}(0,\hat \xi)\tilde \pi_k(y_1),\\
\ea
for some smooth nonsingular matrix-valued functions 
$\alpha$, $\tilde \alpha \in \R^{(n+1)\times(n+1)}$, with
$$
\tilde \alpha^* \alpha=I_{n+1}.
$$

Denoting by $\alpha_j$, $\tilde \alpha_j$ the $j$th columns
of $\alpha$, $\tilde \alpha$, we 
%thus find that $G^\dagger$ factors as
obtain the Floquet-type factorization
\ba\label{floq}
G^\dagger(x,t;y)&=&\Pi(x_1)g^\dagger(x-y,t)\tilde \Pi(y_1)^{*},\\
\label{gxi2star}
g^\dagger(x-y,t)&:=& 
\Big(\frac{1}{2\pi }\Big)^d \int_{\R^{d}} e^{i\xi \cdot (x-y)} 
[g^\dagger_\xi(t)] 
d\xi ,\\
\label{gxistar}
[g^\dagger_\xi(t)]&:=& \phi(\xi)\sum_{j=1}^{n+1}e^{\lambda_j^\dagger(\xi)t}
\alpha_j(0,\hat \xi)\tilde \alpha_j (0,\hat \xi)^*,
\ea
where $\Pi:=(\pi_1,\dots, \pi_{n+1})$,
$\tilde \Pi:=(\tilde \pi_1,\dots, \tilde \pi_{n+1})$, 
and $g$ is a constant-coefficient operator in the sense that
it is invariant under spatial translations.
Recall, by definition, that $\tilde  \Pi^*\Pi=I_{n}$, so that
$\tilde \Pi^*$ is something like a pseudoinverse of $\Pi$
(not ``the'' pseudoinverse, however, except in the self-adjoint
case $\Sigma_0=\tilde \Sigma_0$). 

Indeed, we may factor a bit further, as
\be\label{WB}
g^\dagger= W * K=K*W,
\ee
where
\be\label{W}
W(z,t):=
\Big(\frac{1}{2\pi }\Big)^d \int_{\R^{d}} e^{i\xi \cdot z} 
\phi(\xi)\sum_{j=1}^{n+1}e^{a_j(\xi)t}
\alpha_j(0,\hat \xi)\tilde \alpha_j (0,\hat \xi)^* d\xi 
\ee
denotes the hyperbolic part of the solution operator and
\be\label{K}
K(z,t):=
\Big(\frac{1}{2\pi }\Big)^d \int_{\R^{d}} e^{i\xi \cdot z} 
\phi(\xi)\sum_{j=1}^{n+1}e^{b_j(\xi)t}
\alpha_j(0,\hat \xi)\tilde \alpha_j (0,\hat \xi)^* d\xi 
\ee
denotes the diffusive part, with $b_j(\xi):=|\xi|^2b_j(\hat \xi)$.
This gives a description of the constant-coefficient solution
operator 
$s^\dagger(t)f:= g^\dagger(\cdot, t) * f$ as the composition
$S^\dagger= S_W\dot S_K=S_K\dot S_W$ of 
commuting hyperbolic and parabolic solution operators $S_W(t)f:=W(\cdot,t)*f$,
$S_K(t)f:=K(\cdot,t)*f$, 
and of the Green kernel $g^\dagger$ as 
a {\it linear convection--diffusion wave}
$W*K=S_W(t)K$ as defined in \cite{HoZ1} for general constant-coefficient
systems.

\subsection{Convergence to linear convection--diffusion wave}
Define now $w (x,t):=\tilde \Pi^*(x_1)v^\dagger (x,t)\in \R^{n+1}$.
Evidently, 
\be\label{vres}
v^\dagger (x,t)=\Pi(x_1)w(x,t),
\quad
w= g^\dagger * w_0, 
\quad
w_0 (x):=\tilde \Pi^*(x_1)v_0 (x).
\ee
Denote by 
\be\label{Wdef}
W:=\int_{\R^d}w_0(x)dx =\int_{\R^d}\tilde \Pi(x_1)^* v_0(x)dx
\ee
the total mass of $w_0$.

\begin{theorem}\label{t:behavior}
Assuming (H0)-(H3) and (D1)-(D3), 
for $d\ge 3$, $t\ge 1$,
\ba\label{e:bp3}
\|u(\cdot,t)-\bar{u}-\Pi g^\dagger(\cdot,t)W \|_{L^p} &\le& 
C  (1+t)^{-\frac{d}{2}(1-\frac{1}{p})-\frac{1}{2}} 
(\|v_0\|_{L^1\cap H^k}+\|xw_0\|_{L^1}),\\
\label{e:bH3}
\|u(\cdot,t)-\bar{u}-\Pi g^\dagger(\cdot,t)W \|_{H^k} &\le& 
C  (1+t)^{-\frac{d}{4}-\frac{1}{2}} 
(\|v_0\|_{L^1\cap H^k}+\|xw_0\|_{L^1}),
\ea
where $W\in \R^{n+1}$ is the constant mass vector defined in (\ref{Wdef}).  
Moreover,
\be\label{vbd}
C_1  (1+t)^{-\frac{d}{4}} 
\le \|g^\dagger(\cdot,t) \|_{L^2} \le
C_2  (1+t)^{-\frac{d}{4}},
\ee
so that in general 
$ \|u(\cdot,t)-\bar{u}-\Pi g^\dagger(\cdot,t)W \|_{L^2}<< 
\|\Pi g^\dagger(\cdot,t)W \|_{L^2} $, or
$ u(\cdot,t)-\bar{u}\sim \Pi g^\dagger W$.
\end{theorem}

\begin{proof}
By Proposition \ref{t:behavior1}, (\ref{vres}), and boundedness of $\Pi$
and derivatives, it is sufficient to show that
\ba\nonumber
\|g^\dagger * w_0- g^\dagger W\|_{L^p} 
&\le& 
C  (1+t)^{-\frac{d}{2}(1-\frac{1}{p})-\frac{1}{2}} \|xw_0\|_{L^1},\\
\nonumber
\|g^\dagger * w_0- g^\dagger W\|_{H^k} &\le& 
C  (1+t)^{-\frac{d}{4}-\frac{1}{2}} \|xw_0\|_{L^1},
\ea
or, by Hausdorff--Young's inequality and Parseval's identity, that
\ba\label{e:bp5}
\|\widehat {g^\dagger} (\hat w_0(\xi)- W)\|_{L^q} 
&\le& 
C  (1+t)^{-\frac{d}{2}(1-\frac{1}{p})-\frac{1}{2}} \|xw_0\|_{L^1},\\
\label{e:bH5}
\|(1+|\xi|^k)\widehat {g^\dagger} (\hat w_0(\xi)- W)\|_{L^2(\xi)}
&\le& 
C  (1+t)^{-\frac{d}{4}-\frac{1}{2}} \|xw_0\|_{L^1},
\ea
where $\frac{1}{p}+\frac{1}{q}=1$.
Noting, by the Mean Value Theorem and Hausdorff--Young's inequality, that
$$
 \|\hat w_0(\xi)- W\|_{L^\infty(\xi)} =
 \|\hat w_0(\xi)- \hat w_0(0)\|_{L^\infty(\xi)} 
\le 
 |\xi| \|\partial_\xi \hat w_0 \|_{L^\infty(\xi)}\le
|\xi| \|\partial_\xi \hat w_0 \|_{L^\infty(\xi)} \le 
|\xi|\|xw_0\|_{L^1}, 
$$
we readily obtain (\ref{e:bp5})--(\ref{e:bH5}) from
$$
\|\xi\widehat {g^\dagger} \|_{L^q} \le 
C  (1+t)^{-\frac{d}{2}(1-\frac{1}{p})-\frac{1}{2}}
\quad \hbox{\rm and}\quad
\|\xi (1+|\xi|^k)\widehat {g^\dagger} \|_{L^2(\xi)}
\le 
C  (1+t)^{-\frac{d}{4}-\frac{1}{2}},
$$
as follow by direct computation from representation
(\ref{gxi2star})--(\ref{gxistar}) as, likewise, does (\ref{vbd});
see \cite{HoZ1} for similar computations.
\end{proof}

\section{Discussion and open problems}\label{discussion}

Theorem \ref{t:behavior} shows that the $L^2$-asymptotic behavior
of perturbations $v$ with initial value $v_0$
possessing an $L^1$ first moment
is given by a linear convection-diffusion
wave 
$$
\Pi(x_1) g^\dagger(x,t)M=\Pi(x_1) W*K(x,t)M
$$
with amplitude determined by the modulated
mass $M=\int_{\R^d}\tilde \Pi^*(x_1)v_0(x)dx$, decaying
in $L^2$ at the rate of a heat kernel (Gaussian).
Due to the convective-diffusive structure of $g^\dagger$,
this is essentially all we can say without further assumptions
on the explicit structure of the hyperbolic system (\ref{e:wkb}).
For, geometric effects such as focusing or defocusing of characteristics
can greatly affect the $L^p$ norm of $W*K$ for norms $p>2$, as discussed
in \cite{HoZ1} for the specific case of the wave equation.
It might even be that for sufficiently large $p$ a different part
of the solution dominates behavior.

A brief consideration reveals that the zero eigenspace of $L_0$
spaned by the columns of $\Pi$ consists of tangent directions
along the manifold of possible periodic solutions nearby $\bar u$.
Thus, our description (\ref{floq}) of lowest-order behavior as
the product of $\Pi$ and a solution $g^\dagger *( \tilde \pi v_0)$ 
of a diffusive regularization of a hyperbolic system corresponding 
to (\ref{e:wkb}) can be viewed roughly as a linearized version of
the formal description by WKB approximation, in which, to lowest
order, the solution is approximated by
\be\label{ansatz}
\bar u^{\zeta(x,t)}(\psi(x,t)),
\ee
where $\zeta$ indexes the manifold of nearby traveling waves
and $\psi$ is a scalar phase function with $\grad_x \psi=N\Omega$;
see \cite{Se1,OZ3} for further discussion.
%TODO"
%NOTE: (Explicit connection between WKB and behavior would be nice-
%still mysterious to me- and might resolve these issues in
%low dimensions... -KZ)
%the above is close I think, but not yet sufficiently explicit...

For lower dimensions $d=1$ and $2$ where decay of the linearized
solution is slower, behavior is not expected
to be dominated by its linear part, and indeed we have seen that
the description of the solution as linear part plus error is too
crude even to close a stability argument.
A very interesting direction for further investigation would be
to attempt to encode the lowest-order behavior at a nonlinear level
using (\ref{ansatz}) or a slight modification, so as to eliminate
the largest terms in the nonlinear residual and close the stability
argument.  See, for example, the argument used in \cite{HoZ2}
to obtain stability and behavior of scalar viscous shock fronts
in the critical dimension $d=2$, in which quite similar difficulties
arise.
See also the remarkable work of Schneider \cite{S1,S2,S3} on stability
of patterns in dimensions $d=1,2$, 
in which the nonlinear behavior encoded by modulation
equations likewise plays a crucial role in the analysis by revealing
unexpected cancellation needed to close the stability estimates.

We remark that our way of getting linearized $L^2$ estimates
based on isometry properties is essentially different from, 
and somewhat simpler than, 
either the weighted norm approach of \cite{S1,S2,S3} 
or the one-dimensional pointwise approach of \cite{OZ1}.
This more primitive approach allows us to 
treat cases, as here, for which the low-frequency dispersion
relation is not smooth, as can be expected for general systems
for which convection plays a role.

Finally, we recall that the analysis of \cite{S1,S2,S3} concerns
general multiply periodic waves, i.e., waves 
that are either periodic or else constant in each coordinate direction.
It would be very interesting to consider whether such waves with two
or more periods can arise as solutions of conservation laws,
and if so, what would be the resulting behavior, even at a
formal WKB level.

\bibliographystyle{plain}

\begin{thebibliography}{GMWZ.7}

\bibitem 
[G]{G} R. Gardner, {\it On the structure of the spectra of
periodic traveling waves}, J. Math. Pures Appl. 72 (1993), 415-439.

\bibitem
[He]{He} D. Henry,
{\it Geometric theory of semilinear parabolic equations},
Lecture Notes in Mathematics, Springer--Verlag, Berlin (1981)

\bibitem
[H]{H} I.L. Hwang,
{\it The $L^2$-boundedness of pseudodifferential operators,}
Trans. Amer. Math. Soc. 302 (1987) 55--76.

\bibitem
[K]{K} T. Kato,
{\it Perturbation theory for linear operators},
Springer--Verlag, Berlin Heidelberg (1985).

\bibitem
[HoZ1]{HoZ1}
D. Hoff and K. Zumbrun, 
\textit{Multi-dimensional diffusion waves for the
Navier-Stokes equations of compressible flow}.
Indiana Univ. Math. J. 44 (1995), 603--676.

\bibitem
[HoZ2]{HoZ2}
D. Hoff and K. Zumbrun, 
\textit{Asymptotic behavior of
multi-dimensional scalar viscous shock fronts},
Indiana Math. J. (2000), No. 2, 427--474.

\bibitem
[O]{O} 
M. Oh,
\textit{Numerical study of low-frequency stability of
periodic solutions of van der Waals gas dynamics},
Z. Anal. Anwend.  25  (2006),  no. 1, 1--21.

\bibitem
[OZ1]{OZ1} 
M. Oh and K. Zumbrun, \textit{Stability of periodic 
solutions of viscous conservation laws with viscosity- 
1. Analysis of the Evans function},  
Arch. Ration. Mech. Anal. 166 (2003), no. 2, 99--166.

\bibitem
[OZ2]{OZ2} 
M. Oh and K. Zumbrun, \textit{Stability of periodic 
solutions of viscous conservation laws with viscosity-
Pointwise bounds on the Green function}, 
Arch. Ration. Mech. Anal. 166 (2003), no. 2, 167--196.

\bibitem
[OZ3]{OZ3} M. Oh, and K. Zumbrun,
\textit{Low-frequency stability analysis of periodic 
traveling-wave solutions of viscous conservation laws 
in several dimensions}, 
Journal for Analysis and its Applications, 25 (2006), 1--21.

\bibitem
[Pa]{Pa} A. Pazy, {\it Semigroups of linear operators and applications 
to partial differential equations,} Applied Mathematical Sciences, 44, 
Springer-Verlag, New York-Berlin, (1983) viii+279 pp. ISBN: 0-387-90845-5.

\bibitem
[S1]{S1} G. Schneider, {\it Nonlinear diffusive stability
of spatially periodic solutions-- abstract theorem and higher space
dimensions}, 
Proceedings of the International Conference on Asymptotics 
in Nonlinear Diffusive Systems (Sendai, 1997),  159--167, 
Tohoku Math. Publ., 8, Tohoku Univ., Sendai, 1998. 

\bibitem
[S2]{S2} G. Schneider, 
{\it Diffusive stability of spatial periodic solutions of the 
Swift-Hohenberg equation,} (English. English summary) 
Comm. Math. Phys. 178 (1996), no. 3, 679--702. 

\bibitem
[S3]{S3} G. Schneider, 
{\it Nonlinear stability of Taylor vortices in infinite cylinders,}
Arch. Rat. Mech. Anal. 144 (1998) no. 2, 121--200.

\bibitem
[Se1]{Se1} D. Serre,
{\it Spectral stability of periodic solutions of viscous conservation laws:
Large wavelength analysis}, Comm. Partial Differential Equations 30 (2005),  
no. 1-3, 259--282.

\bibitem
[Z1]{Z1} 
K. Zumbrun, \textit{Multidimensional stability of
planar viscous shock waves}, TMR Summer School Lectures:
Kochel am See, May, 1999, 
Birkhauser's Series: Progress in Nonlinear Differential
Equations and their Applications (2001), 207 pp.

\bibitem[Z2]{Z2}%CIME lecture notes
K. Zumbrun.
{\it Planar stability criteria for viscous shock waves of systems with
real viscosity,}
In {\em Hyperbolic systems of balance laws}, volume 1911 of {\em
  Lecture Notes in Math.}, pages 229--326. Springer, Berlin, 2007.

\bibitem
[ZS]{ZS} 
K. Zumbrun and D. Serre,
\textit{Viscous and inviscid stability of multidimensional 
planar shock fronts,} Indiana Univ. Math. J. 48 (1999) 937--992.

\end{thebibliography}

\end{document}